\documentclass{article}
\usepackage{graphicx} 

\usepackage{amsmath}
\usepackage{amsfonts}
\usepackage{amsthm}
\usepackage{amssymb}
\usepackage[all]{xy}
\newcommand{\XX}{\mathcal{X}}
\newcommand{\Rc}{\mathcal{R}}
\newcommand{\ZZ}{ \mathbb{Z}}
\newcommand{\CC}{\mathbb{C}}
\newcommand{\DD}{\mathbb{D}}
\newcommand{\Lc}{\mathcal{L}}
\newcommand{\CCa}{\mathcal{C}}
\newcommand{\Mod}{\textrm{Mod}}

\newtheorem{theorem}{Theorem}[section]
\newtheorem{proposition}[theorem]{Proposition}
\newtheorem{definition}[theorem]{Definition}
\newtheorem{quest}[theorem]{Question}
\newtheorem{lemma}{Lemma}[theorem]
\newtheorem{conjecture}[theorem]{Conjecture}
\title{Monodromy action on character varieties for Lefschetz pencils}
\author{Ishan Banerjee }

\begin{document}

\maketitle
\abstract{Given a Riemann surface $\Sigma$, let $\Gamma \subseteq \Mod(\Sigma)$ denote the monodromy subgroup of a family of complex curves homeomorphic to $\Sigma,$ arising from a sufficently ample Lefschetz pencil. We establish that the group $\Gamma$ acts with Zariski dense orbits or ergodically on certain character varieties for $\pi_1(\Sigma)$. This answers a version of a conjecture of Katzarkov, Pantev, and Simpson appearing in \cite{KPS}.}
\section{Introduction}

Let $\Sigma$ be a closed Riemann surface of genus $g$ and $G$ a reductive complex algebraic group. 
Let $*\in G.$ Let $\Rc_G(\Sigma, *)$ denote the $G-$representation variety of $\Sigma$,  $$\Rc_G(\Sigma, *  )  := \{ \rho: \pi_1(\Sigma, *) \to G | \rho \textrm{ is a homomorphism}\}.$$
We have a action of $G$ on $\Rc_G(\Sigma, *  )$ by conjugation. We define  $$\XX_G(\Sigma): = \Rc_G(\Sigma) / G,$$ where the quotient is taken in the GIT sense.

We will call $\XX_G(\Sigma)$  the $G-$character variety of $\Sigma$. The mapping class group of $\Sigma$, $\Mod(\Sigma)$ acts on $\XX_G(\Sigma).$ Thus any subgroup $\Gamma \subseteq \Mod(\Sigma)$ also acts on $\XX_G(\Sigma).$ We are interested in the following question.
\begin{quest}
    Given a subgroup $\Gamma \subseteq \Mod(\Sigma),$ does $\Gamma$ act with Zariski dense orbits on $\XX_G(\Sigma)$?
\end{quest}

What is known about this? Let us  mention that using the main Theorem of \cite{PX} (and a little additional work) we can conclude that for an arbitrary connected reductive group $G$, it can be shown that the \emph{entire mapping class group} $\Mod(\Sigma)$ acts with Zariski dense orbits on $\XX_G(\Sigma).$
We also have the following theorem of Katzarkov, Pantev and Simpson.
\begin{theorem}[Theorem A and B of \cite{KPS}]
    Let $n$ be a positive odd integer.
    Let $\Gamma \subseteq \Mod(\Sigma)$ be one of the following three groups:
    \begin{enumerate}
        \item $\Mod(\Sigma)$
        \item A hyperelliptic mapping class group.
        \item The monodromy group of a sufficiently ample Lefschetz pencil in a smooth complex projective surface $X$ , where $H^1(X,\ZZ) =0.$
    \end{enumerate}
    Then $\Gamma$ acts on $\XX_G(\Sigma)$ with dense orbits when $G=GL_{n}(\CC).$
\end{theorem}
They also made the following conjecture.
\begin{conjecture}\label{conj}
Let $G$ be a reductive algebraic group.
Let $\Gamma \subseteq \Mod(\Sigma)$ be either: 
    \begin{enumerate}
    \item the monodromy group of an arbitrary non- isotrivial family of algebraic curves.
    \item The monodromy group of a sufficiently ample Lefschetz pencil of curves in a smooth complex projective surface $X.$
    \end{enumerate}
    Then $\Gamma$ acts on $\XX_G(\Sigma)$ with Zariski dense orbits.
\end{conjecture}
More details will be given on the second kind of subgroup in Section 5 of this paper. We also note that the authors of \cite{KPS} conjecture slightly more, they not only conjecture that the action has Zariski dense orbits but also that for some $\rho \in \XX_g(\Sigma)$ with finite image, the action of (a finite index subgroup) of $\Gamma$ on $T_{\rho}\XX_G(\Sigma)$ has Zariski dense image. We do not study this further question in this paper.

While part (i) of the conjecture is still wide open,we prove a corrected version of part(ii) of the conjecture.
\begin{theorem}\label{main}
    Let $X$ be a smooth projective surface, $D$ a very ample divisor on $X$. Let $k \ge 7$. Let $\Gamma$ be the monodromy group associated to the complete linear system $|kD|,$ equivalently it is the monodromy group associated to a Lefschetz pencil in this complete linear system.
    Let $G$ be a semisimple algebraic group.  Then $\Gamma$ acts with Zariski dense orbits on $\XX_G(\Sigma).$
\end{theorem}
The assumption that $G$ is semisimple is necessary, both Conjectures are false for reductive groups with $G^{ab} \neq 0$. See subsection \ref{ssnec} for further details.

Our approach is completely different from that of \cite{KPS}. They approach the problem of proving Zariski density by finding a specific representation $\rho$ that is fixed by a finite index subgroup of the mapping class group. They then prove that for this $\rho$, the induced action of (a finite index subgroup of) $\Mod(\Sigma)$ on $T_{\rho} \XX_G(\Sigma)$ has Zariski dense orbits. This is enough to establish the result. This approach while quite interesting, is very dependent on finding finite subgroups with specified properties inside the ambient group $G$, which is why the argument only works for $GL_{2n+1}(\CC).$ Our methods on the other hand are more reminiscent of the sewing methods involved in \cite{PX} (along with methods developed in \cite{BS}).

Let us emphasize that the most interesting situation when Theorem \ref{main} holds is when $\pi_1 X$ is a large complicated group. Indeed, if $\pi_1 X =0 $, then the results of \cite{BS} imply that $\Gamma$ is of finite index in the mapping class group. The results of Pickrell-Xia in \cite{PX} (and a little more work to pass from $G$ to a maximal compact subgroup inside $G$) would then imply Theorem \ref{main} in that case.

However if $\pi_1 X$ is a large group, then $\Gamma$ is very much of infinite index in  $\Mod(\Sigma)$, and these groups can have interesting properties. This was explored by the author in \cite{Bpi}, there one can see interesting relations between the groups $H^1(X, \ZZ)$ and $
\Gamma.$

Finally we'd like to note that there are indeed somewhat large subgroups of $\Mod(\Sigma_g)$ that act with non Zariski dense orbits: Saadi \cite{Saadi} produces a filling family of multicurves, whose associated Dehn twists  preserve a rational function.
\subsection{Ergodic version}
Given a compact connected Lie group $K$ one can define the representation variety $\Rc_K(\Sigma, *)$ exactly as above. We can then define $\XX_K(\Sigma): = \Rc_K(\Sigma) / K$ where the quotient is the usual topological quotient. We have:
\begin{theorem}\label{mainerg}
        Let $X$ be a smooth projective surface, $D$ a very ample divisor on $X$. Let $k \ge 7$. Let $\Gamma$ be the monodromy group associated to the complete linear system $|kD|,$ equivalently it is the monodromy group associated to a Lefschetz pencil in this complete linear system.
    Let $K$ be a semisimple compact group.  Then $\Gamma$ acts ergodically on $\XX_K(\Sigma),$ with respect to the standard symplectic measure (arising from the Goldman symplectic form).
\end{theorem}

A major advantage of our approach over that of \cite{KPS} is that the machinery is flexible enough to prove results in both the ergodic and algebraic setting. In fact our proofs will be almost identical in both settings.
\subsection{Necessity of semisimplicity}\label{ssnec}
In this subsection we will establish that Conjecture \ref{conj} is false for any reductive group that is not semi-simple.

Let us recall an important Theorem of \cite{KPS} which will be used over and over in this paper.
\begin{theorem}[Theorem 2.2 of \cite{KPS}]
    Let $\Gamma$ be a discrete group acting on an affine variety $X$. Then 
    $\CC(X)^{\Gamma} = \CC$ if and only if $\Gamma$ has a dense orbit.
\end{theorem}

In light of this theorem, we can study invariant functions instead of trying to directly establish that a certain orbit is dense.
We recall also the well known fact that a group action on a probability space $(X, \mu)$ is ergodic if and only if there are no nontrivial invariant $L^2$ functions. Thus the properties of an action being ergodic and being Zariski dense are structurally very similar, this will be used repeatedly.

\begin{proposition}\label{notreductive}
Let $\Sigma$ be a closed Riemann surface of genus $\ge 1$. Let $G$ be a reductive group with abelianisation $A$, with $\dim A >0$. Let $\Gamma \subseteq \Mod(\Sigma)$. Let $V \subseteq H_1(X,\ZZ)$ denote a proper symplectic subspace. Assume that $\Gamma$ maps to $Sp(V) \subseteq Sp(H_1(\Sigma))$ under the Torelli map. Then $\CC(\XX_G(\Sigma))^{\Gamma} \neq \CC.$ As a consequence, there are no Zariski dense orbits.
\end{proposition}
\begin{proof}
    We have a $\Gamma$ equivariant map $p :\XX_G(\Sigma) \to \XX_A(\Sigma)= H_1(\Sigma, \ZZ) \otimes_{\ZZ} A.$ By our assumption we have a $\Gamma$ invariant projection $p_0:H_1(\Sigma, \ZZ)\otimes_{\ZZ} A \to V^{\perp} \otimes_{\ZZ} A.$ Let  $f \in \CC(V^{\perp} \otimes_{\ZZ} A)$ be a nonconstant function. Then $f \circ p_0 \circ p$  is a nonconstant element in $\CC(\XX_G(\Sigma))^{\Gamma}.$
\end{proof}
We note that there are many examples of monodromy groups satisfying the assumptions of Proposition \ref{notreductive}. As an example one may take any smooth projective surface $X$, with $H^1(X,\ZZ) \neq 0$ and take the monodromy group associated to a Lefschetz pencil in $X$.
 
\subsection{Overview}
\begin{enumerate}
    \item In the second section, we define the notion of the closure of a subgroup with respect to a group action and establish some useful properties related to this notion.
    \item In the third section we define several induction lemmas, these lemmas give us various inductive criteria for subgroups of the mapping class group to act with Zariski dense orbits on relative representation varieties.
    \item In the fourth section we discuss how to deal with some starting and ending points of our induction, we deal with the case of a torus with boundary and the case of framed mapping class groups.
    \item In the fifth section we discuss how to adapt the machinery of \cite{BS} to the non-simply-connected situation. While not everything goes through, the machinery can be modified to prove our desired Theorems.
\end{enumerate}
\subsection{Conventions}
In this paper henceforth, $G$ will always denote a complex connected semisimple algebraic group. We will sometimes drop the $G$ from our notation, i.e. we will denote the $G$ character variety by $\XX(\Sigma)$ as opposed to $\XX_G(\Sigma)$ if there is no risk of confusion. 
The letter $K$ will henceforth be used for a compact connected semisimple Lie group.
We will not drop $K$ from our notation, $K$ character varieties will always be denoted $\XX_K(\Sigma).$

Given a representation $\rho: \pi_1(\Sigma, *) \to G$, an unbased loop $\gamma \subseteq \Sigma$ and a conjugacy class $c \subseteq G$ we will say that $\rho(\gamma) \in c$ if for some choice of based loop $\tilde \gamma$
 isotopic to $\gamma,$ $\rho(\tilde \gamma) \in C.$

\subsection{Acknowledgments}
I would like to thank Nick Salter for many helpful suggestions on a preliminary draft of this paper. I would like to thank Jake Huryn for many helpful discussions about this topic. Finally, I would like to thank Daniel Litt for telling me about this problem.

\section{Closure with respect to a group action}
\begin{definition}
    Let $\Gamma$ be a discrete countable group acting on a probability space $(X,\mu)$ in a measure preserving way.   Let $\Gamma_0 \subseteq \Gamma$. The closure of $\Gamma_0$ in $\Gamma$ with respect to the above action is the group $$\overline{\Gamma_0} = \{g \in \Gamma |L^2(X)^{\Gamma_0} \subseteq L^2(X)^g\}.$$
\end{definition}
We note that we can analogously define the closure with respect to an algebraic action.
\begin{definition}[Closure with respect to an algebraic action]
    Let $\Gamma$ be a discrete group acting on a complex affine variety $X$. Let $\Gamma_0 \subseteq \Gamma$. The closure of $\Gamma_0$ in $\Gamma$ with respect to the above action is the group $$\overline{\Gamma_0} = \{g \in \Gamma |\CC(X)^{\Gamma_0} \subseteq \CC(X)^g\}.$$
\end{definition}

In other words, the closure of $\Gamma_0$ is the largest subgroup of $\Gamma$ such that a $\Gamma_0$ invariant function is also $\overline{\Gamma_0}$ invariant.

We note that the above definitions depend on the group action of $\Gamma$ on $X$ and is not intrinsic to the subgroup $\Gamma_0$.


\begin{lemma}\label{changeofclosurealg}
    Let $G$ be a complex reductive group.
    Let $\Sigma_0 \subseteq \Sigma$ be a subsurface with exactly one boundary component $b$. Let $* \in b$. We assume $\Sigma^c = \Sigma \setminus int( \Sigma_0)$ is of genus $\ge 2$.
Let $\Gamma _0 \subseteq \Mod(\Sigma_0) \subseteq \Mod(\Sigma,*).$ Given $g \in G$, let  $\Rc_g(\Sigma_0, *)$ denote the relative representation variety of $\Sigma_0$ with boundary monodromy $g$. i.e. $$\Rc_g(\Sigma_0, *) = \{\rho : \pi_1(\Sigma_0,*) \to G | \rho(b) =g\}.$$

Assume that for general $g$, $\Gamma_0$ acts with dense orbits on the relative character variety $\Rc_g(\Sigma_0, *)$. Then $\overline\Gamma_0 = \overline{\Mod(\Sigma_0)}$ with respect to the group action on $\Rc_G^{\partial}(\Sigma).$
\end{lemma}

\begin{proof}
    It suffices to establish that a $\Gamma_0$ invariant function $f \in \CC(\Rc(\Sigma,*))$ is  $\Mod(\Sigma_0)$ invariant as well.
    We have a $\Mod(\Sigma_0)$ invariant restriction map $\psi: \Rc(\Sigma, *) \to \Rc(\Sigma^c, *)$ whose fibers are precisely the relative representaion varieties $\Rc_g(\Sigma_0, *)$. By our assumptions, any $\Gamma_0$ invariant function on $\Rc(\Sigma,*)$ is constant on the general fiber of $\psi$ and hence must factor through $\psi.$ Thus $f$ is  $\Mod(\Sigma_0)$ invariant.
\end{proof}

\begin{lemma}\label{changeofclosureerg}
    Let $K$ be a compact connected reductive group.  
    Let $\Sigma_0 \subseteq \Sigma$ be a subsurface with exactly one boundary component $b$. Let $* \in b$. We assume $\Sigma^c = \Sigma \setminus int( \Sigma_0)$ is of genus $\ge 2$.
Let $\Gamma _0 \subseteq \Mod(\Sigma_0) \subseteq \Mod(\Sigma,*).$ Given $k \in K$, let  $\Rc_{k,K}(\Sigma_0, *)$ denote the relative representation variety of $\Sigma_0$ with boundary monodromy $k$. i.e. $$\Rc_{k,K}(\Sigma_0, *) = \{\rho : \pi_1(\Sigma_0,*) \to K | \rho(b) =k\}.$$

Assume that for a.e. $k$, $\Gamma_0$ acts with dense orbits on the relative character variety $\Rc_k(\Sigma_0, *)$. Then $\overline\Gamma_0 = \overline{\Mod(\Sigma_0)}$ with respect to the group action on $\Rc_K^{\partial}(\Sigma).$
\end{lemma}
\begin{proof}
    This follows mutatis mutandis from the proof of the previous lemma.
\end{proof}
Let us now make a few observations on what these Lemmas gives us. Let $\Sigma$ be a closed surface, with a base point $*$. Let $\Gamma \subseteq \Mod(\Sigma, *)$ be a subgroup.  Let $\Sigma_0$ be a surface with one boundary component containing $*$. 
The combination of the above lemmas allows us to conclude that:
\begin{enumerate}
    \item 
   If $\Gamma \cap \Mod(\Sigma_0)$ acts with Zariski dense orbits on the relative representation varieties $\Rc_g( \Sigma_0, *)$ for general $g \in G$ then $\overline{\Gamma}$ contains the image of  $\Mod(\Sigma_0)$ in $\Mod(\Sigma
    ,*),$ here the closure is with respect to the $\Mod(\Sigma, *)$ action on $\Rc(\Sigma,*).$
    \item   If $\Gamma \cap \Mod(\Sigma_0)$ acts ergodically on the relative representation varieties $\Rc_{k,K}( \Sigma_0, *)$ for general $k \in K$ then $\overline{\Gamma}$ contains the image of $\Mod(\Sigma_0)$ in $\Mod(\Sigma
    ,*),$ here the closure is with respect to the $\Mod(\Sigma,*)$ action on $\Rc_K(\Sigma,*).$
\end{enumerate}

These two observations will be helpful to us in computing closures of various subgroups, if we have a group $\Gamma \subseteq \Mod(\Sigma)$ such that $\Gamma \cap \Mod(\Sigma_0)$ acts with Zariski dense orbits on the relevant representation varieties we can conclude that $\overline\Gamma = \overline{<\Gamma, \Mod(\Sigma_0)>},$ and $<\Gamma, \Mod(\Sigma_0)>$ is often a more manageable group than $\Gamma$ itself. 

\section{Induction}
In this section we will establish a few induction lemmas that enable us to conclude that certain groups act with Zariski dense orbits or ergodically on relative representation varieties of larger and larger subsurfaces.

\subsection{Gluing onto the same boundary (algebraic)}
      Let $\Sigma $ be a surface of genus $\ge 2$ with a distinguished boundary component $b,$ and $n$ other boundary components $\delta_1, \dots, \delta_n $. We glue an arc $l$ to $\Sigma$, with bouth end points being glued to $b.$ We can the thicken $\Sigma \cup l$ into a surface $\Sigma^+,$ by thickening the arc $l$ in to a rectangular strip $S$.     Let $P$ denote the union of $b$ and  $S$. The subset  $P$ is homotopy equivalent to a pair of pants. The resulting surface $\Sigma^+$ then has two new distinguished boundary components $b_1$ and $b_2$. Let $d_1 ,\dots, d_n$ be $n$ conjugacy classes in $G$. Let $g_1 \in G$ and let $ c_2$ be a conjugacy class in $G$. Let $* \in \Sigma^+$ be a point in $b_1 \cap b$.
    
    We first define some variants of representation varieties.
    \begin{definition}
      
        Let $\Rc^{\partial} (\Sigma^+,*)$ be the representation variety of all representations $\rho: \pi_1(\Sigma^+, *) \to G$ such that:
        \begin{enumerate}
        \item $\rho(b_1) =g_1.$
        \item  $\rho(b_2) \in c_2$. 
        \item For $1\le i \le n$, we have $\rho(\delta_i) \in d_i.$
        \end{enumerate}
    \end{definition}
    \begin{definition}
            Let $g \in G$. Let $\Rc_g^{\partial} (\Sigma, *)$ denote the representation variety of all representations $\rho: \pi_1 (\Sigma,*) \to G$ such that: 
    \begin{enumerate}
        \item $\rho(b) =g$,
        \item For $1\le i \le n$, we have $\rho(\delta_i) \in d_i$.
    \end{enumerate}
    \end{definition}

    We assume in this subsection that for all $g$ (and for fixed $d_1 ,\dots, d_n$), we have already proven that $\Mod(\Sigma)$ acts on $\Rc_g^{\partial}(\Sigma, *)$ with Zariski dense orbits.
\begin{definition}
        Let $\Rc^{\partial}(P, *)$ denote the representation variety of all representations $\rho: \pi_1(P, *) \to G$ such that: 
    \begin{enumerate}
        \item $\rho(b_1) = g_1$ 
        \item $\rho (b_2) \in c_2.$
    \end{enumerate}
\end{definition}

    We note that $\Rc^ \partial(P, *) \cong c_2$, with the isomorphism
     given as follows. Take an arc $a$ in $P$ going from $*$ to $b_2$  moving along $b$. Let $\tilde b_2$ denote the loop obtained by following $a$, going once around $b_2,$ and then coming back via $a$. The map $\varphi: \Rc^ \partial(P, *) \to  c_2$ given by $\varphi(\rho) = \rho (\tilde b_2)$ is an isomorphism.

     \begin{figure}
         \centering
         \includegraphics[width=0.8\linewidth]{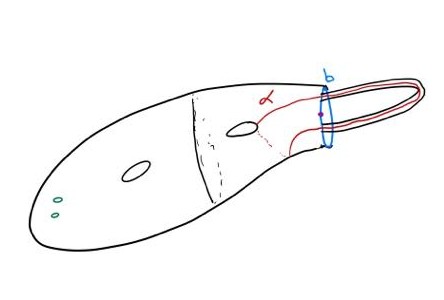}
         \caption{A depiction of the situation of Lemma \ref{indsamebound}}
         \label{fig:samebound}
     \end{figure}
     
\begin{lemma}\label{indsamebound}
    Let $\alpha \subseteq \Sigma^+$ be a loop  as depicted in figure \ref{fig:samebound}, it goes through the attached strip exactly once.
The group $<\Mod(\Sigma), T_{\alpha}>$ acts with Zariski dense orbits on $\Rc^{\partial}(\Sigma^+, *).$
\end{lemma}
\begin{proof}
    Suppose for the sake of contradiction that there is a nonconstant $<\Mod(\Sigma), T_{\alpha}>$  invariant function $f \in \CC(\Rc^{\partial}(\Sigma^+, *)).$

  We have a restriction map $\psi: \Rc^{\partial}(\Sigma^+,*) \to\Rc^{\partial}(P, *),$ this map is $\Mod(\Sigma)$ invariant. Each fiber of $\psi$ is equivariantly isomorphic to   $\Rc_g^{\partial}(\Sigma, *),$ for $g = \rho(b),$ and by our assumption $\Mod(\Sigma)$ acts with Zariski dense orbits on these varieties. Thus $f$ must be constant on these fibers. The map $\psi$ is also dominant. Thus the rational function $f$ must be of the form $\bar f \circ \psi,$ for some $\bar f \in \CC(\Rc^{\partial}(P)).$ 
    We will now prove that $\bar f$ is constant. Recall that $\Rc^{\partial}(P) \cong c_2$ and thus $\bar f$ can be viewed as a rational function on $c_2$. We will prove that this $\bar f$ is in fact conjugation invariant, and hence constant.
    
    Let $\tilde \alpha\in \pi_1(\Sigma^+ , *)$ denote the loop obtained by following $a$ till $a$ intersects $\alpha$ going once around $\alpha$ and then coming back to $*$ along $a$.  We note that $T_{\alpha} \cdot \tilde b_2 = \tilde {\alpha} \tilde b_2 \tilde {\alpha}^{-1}.$ Let $x \in c_2$ be a general element. Importantly, there is an open set $U \subseteq G$, such that for any $y \in U$, there exists $\rho \in \Rc(\Sigma^+, *)$ satisfying $\rho(\tilde \alpha) =y$ and $\rho(\tilde b_2) = x$. This implies  that $$ \bar f (yxy^{-1})= f(T_{\alpha} \cdot \rho)= f(\rho) = \bar f(x).$$ Thus, $\bar f$ is constant and so is $f$.
\end{proof}
\subsection{Gluing on to the same boundary component (ergodic)}
In this subsection, we discuss gluing onto the same boundary component in  the case of $K$ representation varieties. 
Let $\Sigma, \Sigma^+ , P,$ etc. be as it was in the above subsection.Let $d_1 ,\dots, d_n$ be $n$ conjugacy classes in $K$. Let $c_2$ be a conjugacy class in $K$. Let $k_1 \in K.$ Let $* \in \Sigma^+$ be a point in $b_1 \cap b$.
   
    \begin{definition}
        Let $\Rc^{\partial}_K (\Sigma^+,*)$ be the representation variety of all representations $\rho: \pi_1(\Sigma^+, *) \to K$ such that:
    \begin{enumerate}
    \item $\rho(b_1) =k_1.$
    \item  $\rho(b_2) \in c_2$. 
    \item For $1\le i \le n$, we have $\rho(\delta_i) \in d_i.$
    \end{enumerate}

    \end{definition}
   \begin{definition}
           Let $k \in K$. Let $\Rc_{k,K}^{\partial} (\Sigma, *)$ denote the representation variety of all representations $\rho: \pi_1 (\Sigma,*) \to K$ such that: 
    \begin{enumerate}
        \item $\rho(b) =k$,
        \item For $1\le i \le n$, we have $\rho(\delta_i) \in d_i$.
    \end{enumerate}
   \end{definition}

    We assume in this subsection that for all $k$, we have already proven that $\Mod(\Sigma)$ acts on $\Rc_{k,K}^{\partial}(\Sigma, *)$ ergodically.
\begin{definition}
        Let $\Rc^{\partial}(P, *)$ denote the representation variety of all representations $\rho: \pi_1(P, *) \to K$ such that:
    \begin{enumerate}
        \item  $\rho(b_1) = k_1$
        \item $\rho (b_2) \in c_2.$
    \end{enumerate}
\end{definition}
 We note that $\Rc^ \partial(P, *) \cong c_2$, with the isomorphism
     given as follows. Take an arc $a$ in $P$ going from$*$ to $b_2$  moving along $b$. Let $\tilde b_2$ denote the loop obtained by following $a$, going once around $b_2$ and then coming back via $a$. The map $\varphi: \Rc^ \partial(P, *) \to  c_2 $ given by $\varphi(\rho) = \rho (\tilde b_2)$ is an isomorphism.

\begin{lemma}\label{indsamebound}
The group $<\Mod(\Sigma), T_{\alpha}>$ acts with dense orbits on $\Rc^{\partial}_K(\Sigma^+, *).$
\end{lemma}
\begin{proof}
    This follows mutatis mutandis from the proof of the previous lemma, we wil only sketch the proof. Any  $L^2$ function $f$, invariant under the given group, must factor through $\psi : \Rc^{\partial}(\Sigma^+ , *) \to \Rc^{\partial}(P, *)$ (as an $L^2$ function). We can then write $f = \bar f \circ \psi$. We then show, exactly as in the previous lemma, that $\bar f$ is conjugation invariant.
\end{proof}
 

\subsection{Gluing on a torus (algebraic)}\label{gluetor}
    Let $\Sigma $ be a surface of genus $\ge 1$ with a distinguished boundary component $b_0$ and $n$ other boundary components $\delta_1, \dots, \delta_n$. Let $\Sigma^+$ be the surface formed by gluing a two holed torus to $\Sigma$ with one boundary glued to $b_0$. We will call the other boundary component of the torus $b$. Let $* \in \Sigma^+$ be a point in $b$. Let $*_0$ be a point in $b_0$. Let $d_1, \dots, d_n$ be $n$ conjugacy classes in $G$. 
    
    \begin{definition}
        Let $g \in G$. Let $\Rc^{\partial}_g (\Sigma^+,*)$ be the representation variety of all representations $\rho : \pi_1(\Sigma^+, * ) \to G$ such that:
    \begin{enumerate}
        \item $\rho(b) =g$
        \item $\rho(\delta_i) \in d_i$
    \end{enumerate}
    
    \end{definition}

\begin{definition}
        Let $g_0 \in G$. Let $\Rc^{\partial}_{g_0}(\Sigma,*_0)$ be the set of representations $\rho: \pi_1(\Sigma, *_0) \to G$ such that:
    \begin{enumerate}
        \item $\rho(b_0) =g_0$
        \item $\rho(\delta_i) \in d_i$
    \end{enumerate}
\end{definition}

    We assume in this subsection that for general $g_0$, $\Mod(\Sigma)$ acts with dense orbits on  $\Rc^{\partial}_{g_0}(\Sigma,*_0)$.

    Given $g \in G,$ let $\Rc_g(T,*)$ denote the set of representations $\rho:\pi_1(T, *) \to G$ such that $\rho(b) =g.$ Importantly, we make no restrictions on the behaviour of $\rho$ on $b_0.$
    
    Let $x,y \in \pi_1(T, *)$ be two loops as depicted in Figure \ref{fig:indpm}.
    We note that the map $\Rc_g(T ,*) \to G \times G$ given by $\rho \mapsto (\rho(x), \rho(y))$ is an isomorphism.

\begin{figure}
    \centering
    \includegraphics[width=0.8\linewidth]{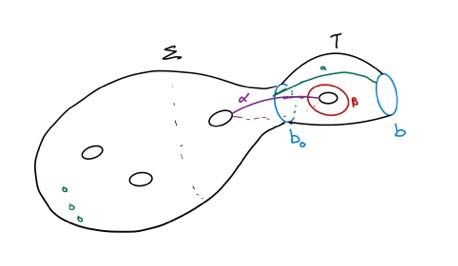}
    \caption{The figure depicts the situation of Lemma \ref{indtor}.}
    \label{fig:placeholder}
\end{figure}

\begin{lemma}\label{indtor}

    Let $\alpha, \beta \subseteq \Sigma^+$ be simple closed curves such that:
    \begin{enumerate}
        \item $\alpha\cap T $ is a single arc.
        \item $\alpha$ intersects $x$ once transversely and does not intersect $y$.
        \item $\beta$ is a meridianal loop in $T$, as depicted in the figure.
        \item $\beta $ intersects $y$ once transversely.
    \end{enumerate}
    The group $ \overline{<\Mod(\Sigma) , T_{\alpha} , T_{\beta}>}$ acts with Zariski dense orbits on $\Rc^{\partial}_g (\Sigma^+,*).$
\end{lemma}
 \begin{proof}
    Let $\tilde \alpha \in \pi_1(\Sigma^+, *)$ denote the loop obtained by following $x$ until it intersects $\alpha$ going around $\alpha$ once and then going back the way we came. We note that $T_{\alpha} \cdot x = \tilde \alpha x$. Similarly $T_{\beta}\cdot y = \tilde \beta y$ where $\tilde \beta$ is defined analogously (with $x$ replaced by $y$).

     Let $\psi: \Rc_g^{\partial}(\Sigma^+, *) \to \Rc_g(T, *)$ denote the restriction map, this map is dominant and $\Mod(\Sigma_0)$ invariant. We will now describe the fibers of this map. 
     Let $a_0$ denote a fixed arc in $T$ joining $* \to *_0$.  Let $\tilde b_0 \in \pi_1(T, *)$ be the loop obtained by following $a_0$ going around $b_0$ and coming back along $a_0$.  Let $\Sigma_1$ denote a small thickening of $\Sigma \cup a$, it is homeomorphic to $\Sigma$ and for any $g_0 \in G$, we have a $\Mod(\Sigma) \cong \Mod(\Sigma_1)$ equivariant isomorphism $\Rc^{\partial}_{g_0}(\Sigma, *_0) \cong \Rc^{\partial}_{g_0}(\Sigma_1, *),$ where the latter variety is the set of representations $\rho:\pi_1 (\Sigma_1, *) \to G$ such that:
     \begin{enumerate}
         \item $\rho(\tilde b_0) = g_0$
         \item $\rho(\delta_i) \in d_i$
     \end{enumerate}
        
    Let $\rho \in \Rc_g(T, *)$.  The fiber $\psi^{-1}(\rho) $ is precisely $\Rc^{\partial}_{\rho (\tilde b_0)}(\Sigma_1, *) \cong \Rc^{\partial}_{\rho(\tilde b_0)} (\Sigma, *_0)!$ We note that for general $\rho$, $\rho(\tilde b_0)$ is a general element in $G$. Since $\Mod(\Sigma)$ acts with dense orbits on  $\Rc_{\rho(\tilde b_0)} (\Sigma, *_0),$ it acts with dense orbits on the fibers. Thus any  $ \overline{<\Mod(\Sigma) , T_{\alpha} , T_{\beta}>}$ invariant rational function must be of the form $\bar f \circ \psi$. 
    
    We will now prove that such a $\bar f$ must be constant. Since $f$ is also $T_{\alpha}$ invariant, we have that $\bar f(\rho(x), \rho(y))$ is also equal to $\bar f( \rho(\tilde \alpha)\rho(x) ,\rho(y)).$ We then note that for  general $a,b \in G $ the group generated by the set  $\{\rho(\tilde \alpha) | \rho(x) =a , \rho(y) =b\}$ is $G$ (this is true even when $\Sigma$ has genus 1). Thus we obtain $\bar f(\rho(x), \rho(y)) = \bar f(G\cdot\rho(x), \rho(y))$ and so $\bar f$ does not depend on the first coordinate. Similarly, we can use the $T_{\beta}$ invariance to conclude that it does not depend on the second coordinate as well.
 \end{proof}   
\subsection{Gluing on a torus (ergodic)}\label{gluetorerg}
    Let $\Sigma ,\Sigma^+, b,$ etc. be as in the previous section.
    Let $d_1, \dots, d_n$ be $n$ conjugacy classes in $K$.
    \begin{definition}
            Let $k \in K.$ Let $\Rc^{\partial}_K (\Sigma^+,*)$ be the representation variety of all representations $\rho : \pi_1(\Sigma^+, * ) \to K$ such that:
    \begin{enumerate}
        \item $\rho(b) =k$
        \item $\rho(\delta_i) \in d_i$
    \end{enumerate}
    \end{definition}

\begin{definition}
        Let $k_0 \in K$. Let $\Rc_{k_0,K}(\Sigma,*_0)$ be the set of representations $\rho: \pi_1(\Sigma, *_0) \to K$ such that:
    \begin{enumerate}
        \item $\rho(b_0) =k_0$
        \item $\rho(\delta_i) \in d_i$
    \end{enumerate}
\end{definition}

    We assume in this subsection that for a.e. $k_0$, $\Mod(\Sigma)$ acts ergodically on  $\Rc_{k_0, K}(\Sigma,*_0)$.

    Given $k \in K,$ let $\Rc_{k,K}(T,*)$ denote the set of representations $\rho:\pi_1(T, *) \to K.$
    Let $x,y \in \pi_1(T, *)$ be two loops as depicted in figure \ref{fig:indpm}.
    We note that the map $\Rc_{k,K}(T ,*) \to K \times K$ given by $\rho \mapsto (\rho(x), \rho(y))$ is an isomorphism.

    Let $\alpha, \beta \subseteq \Sigma^+$ be as in the previous subsection. Similarly, let $\tilde \alpha , \tilde \beta \in \pi_1(\Sigma^+, *)$ be as in the previous subsection.

    \begin{lemma}\label{glutorerg}
    The group $ \overline{<\Mod(\Sigma) , T_{\alpha} , T_{\beta}>}$ acts ergodically on $\Rc^{\partial}_K (\Sigma^+,*).$
\end{lemma}
\begin{proof}
    This follows mutatis mutandis from the proof of Lemma \ref{indtor}.
\end{proof}
    
\subsection{$\pm 1$ curve}

 We first begin the definition of a $\pm 1$ curve (read as plus one and minus one curve) with respect to a handle attachment.

 \begin{definition}
     Let $\Sigma$ be a surface. Let $\Sigma^+$ be a surface formed by attaching a handle $S$ to $\Sigma$. Let the two boundary components of $S$ be $(\partial S)_1$ and $(\partial S)_2.$ 
     We allow the two boundary components of $S$ to be attached to either the same or different boundary components of $\Sigma$.

     A curve $\alpha \subseteq \Sigma^+$ is in $\pm 1$ position with respect to $S$ if:
     \begin{enumerate}
         \item The intersection  $\alpha \cap S$ consists of two arcs $a_1$ and $a_2.$
         \item The arc $a_1$ enters the strip by $(\partial S)_1$ and leaves by $(\partial S)_2$.
         \item The arc $a_2$ enters the strip by $(\partial S)_2$ and leaves by $(\partial S)_1.$
         \item $\Sigma \setminus \alpha$ is connected.
     \end{enumerate}
     
 \end{definition}
 An example of a $\pm 1$ curve is given in Figure \ref{fig:pmcurve}.
\begin{figure}
    \centering
    \includegraphics[width=0.8\linewidth]{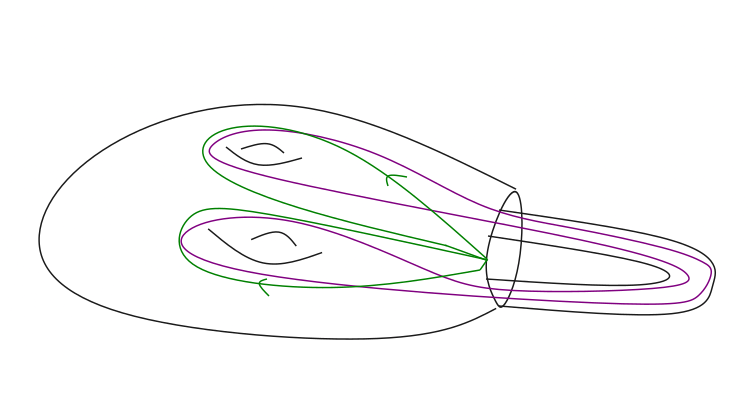}
    \caption{The figure depicts a violet curve that is in $\pm 1$ position with respect to the rectangular strip. The green curves are such that a neighbourhood of the two green curves along with the strip contains the red curve.}
    \label{fig:pmcurve}
\end{figure}
 
\begin{lemma}\label{pmunique}
    Let $\Sigma$ be a curve of genus $\ge 2 $. Let $\Sigma^+$ be an arc formed by attaching a trip $S$ to $\Sigma.$ Let $\alpha, \beta$ be two curves in $\pm 1$ position with respect to $S.$ Then there exists $g \in \Mod(\Sigma)$ such that $g \cdot \alpha = \beta.$
\end{lemma}
\begin{proof}
    This follows from the fact that the mapping class group $\Mod(\Sigma)$ acts transitively on pairs of arcs that do not separate $\Sigma$, this is in turn a consequence of the change of coordinates principle.
\end{proof}
\begin{lemma}\label{disk3h1b}
    Let $\Sigma$ be a surface. Let $\Sigma^+$ be an arc formed by attaching a handle $S$ to the same boundary component. Then $\alpha \cup S$ is contained in a disk with three holes.
\end{lemma}

\begin{proof}
    Let $l_1, l_2$ denote the two disjoint arcs such that $\alpha \cap \Sigma = l_1 \cup l_2.$
    Let $l$ be an arc disjoint from $l_1 , l_2$ joining the two boundary components of $\partial S$
     Then $\alpha \cup S$ is contained in $S \cup l_1 \cup l_2 \cup l$, and when we thicken the latter we obtain a disk with three holes.
\end{proof}
\begin{lemma}\label{disk3h2b}
    Let $\Sigma$ be a surface. Let $\Sigma^+$ be an arc formed by attaching a handle $S$ to two different boundary components. Then $\alpha \cup S$ is contained in a disk with three holes. 
\end{lemma}
\begin{proof}
    This follows mutatis mutandis from the proof of the previous lemma.
\end{proof}
\begin{lemma}\label{pmtwist}
    Let $\Sigma,S, \Sigma^+$ be as above, where $S$ is allowed to be attached to one boundary component or two. Let $\alpha$ be a curve in $\pm 1$ position. Let $a$ denote an arc in $S$ going across the strip. We note that by lemma \ref{disk3h1b} the strip $S$ and $\alpha$ are contained inside a thrice punctured disk which we will denote $D_3$. We depict this disk in Figure \ref{fig:3holedisk}. Then $T_{\alpha} \cdot a = [z_1,z_2]a$ where $z_1,z_2$ are as depicted in the figure. 
\end{lemma}
\begin{proof}
 The arc $a$ joins the outer boundary of the disk with one of the inner boundary components. The element $T_{\alpha}$ can then be understood as an element of $\Mod(D_3).$ The new arc $T_{\alpha} \cdot a$ is then as depicted in Figure \ref{fig:3holedisk}. One can see that $T_{\alpha} \cdot a$  is homotopic to $[z_1,z_2]a$.
\end{proof}
\begin{figure}
    \centering
    \includegraphics[width=\linewidth]{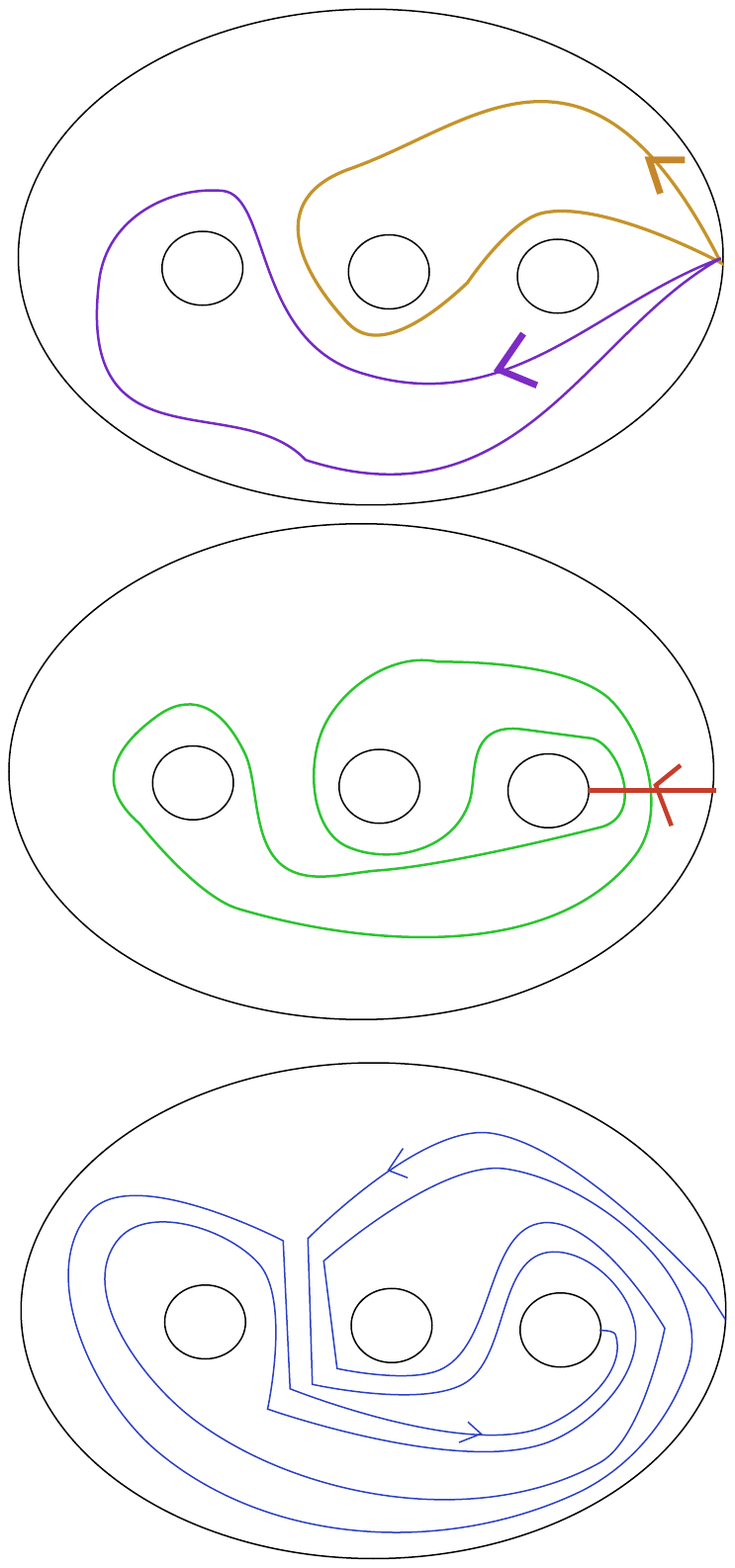}
    \caption{On the top, we  have a disk with 3 holes. In it is a yellow loop $z_1$ and a purple loop $z_2$. In the center we have the same disk with a green loop $\alpha$ and a red arc $a$. On the bottom we have the same disk with $T_{\alpha} \cdot a$ depicted in blue. }
    \label{fig:3holedisk}
\end{figure}
\begin{lemma}\label{indpm}
Let $\Sigma$ be a surface of genus $\ge 4$ with a single boundary component $b_0$. Let $T$ be a two holed torus. We glue $T$ to $\Sigma$ along  one boundary component $b_0$, to obtain a surface $\Sigma^+$. The remaining boundary component of $\Sigma^+$  will be denoted $b$.
Let $\alpha, \beta$ be two curves in $\Sigma^+$ that satisfy:
\begin{enumerate}
    \item $\alpha $ intersects $x$ transversely at one point and is disjoint from $y$.
    \item Let $\Sigma_1$ denote a thickening of $\Sigma \cup \alpha$, it has two boundary components and furthermore $\Sigma^+$ is formed by gluing a strip $S$ to $\Sigma_1$ (the boundaries of $S$ are glued on to distinct different boundary components of $\Sigma_1$). We assume $\beta$ is in $\pm 1$ position with respect to $(\Sigma, S)$. We note that  $\beta$ intersects $y$ twice with opposite orientation.
\end{enumerate}

The figure \ref{fig:indpm} depicts the  above situation.

Given $g \in G$, we let $\Rc_g(\Sigma^+, *)$ be as in subsection \ref{gluetor}.

Then for all $g\in G$, $<\Mod(\Sigma), T_{\alpha}, T_{\beta}>$ acts Zariski densely on $\Rc_g(\Sigma^+, *).$
\end{lemma}

\begin{figure}
    \centering
    \includegraphics[width=\linewidth]{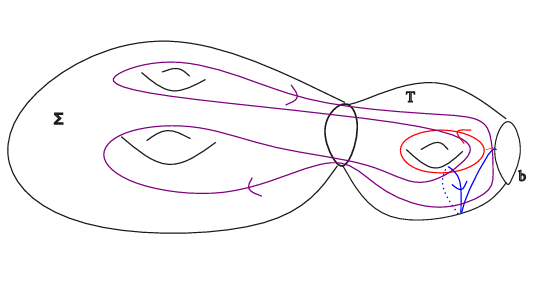}
    \caption{This figure depicts the situation of Lemma \ref{indpm}, the red loop is $x$, the blue loop is $y$, and the purple curve is $\beta.$}
    \label{fig:indpm}
\end{figure}
\begin{proof}
    This is similar to the proof of Lemma \ref{indtor}. Let $f$ be a $<\Mod(\Sigma), T_{\alpha}, T_{\beta}>$ invariant rational function on $\Rc_g(\Sigma^+ , *).$
    Let $\psi: \Rc(\Sigma^+ , *) \to \Rc_g(T, *) \cong  G \times G$ denote the restriction map. Now just as in the proof of Lemma \ref{indtor}, $f$ must factor through $\psi$ since $f$ is $\Mod(\Sigma)$ invariant. 

    Again just as in the proof of Lemma \ref{indtor}, we can use the fact that $f$ is $T_{\alpha}$ invariant to conclude that it in fact factors through the map  $\phi: \Rc(\Sigma^+ , *) \to G $ given by $\phi(\rho) = \rho(y).$ 
    
    Now let $\Sigma_1$ be a thickening of $  \Sigma \cup \alpha$. It is a surface of the same genus as $\Sigma$ with two boundary components. The surface $\Sigma^+$ is formed by gluing a strip $S$ to $\Sigma_1$. The loop $y$ intersects this strip once in an arc perpendicular to the strip.
    The curve $\beta$ is in $\pm 1$ position with respect to the surface $\Sigma_1$ and strip $S.$ Thus by Lemma \ref{disk3h1b}, the strip $S$ and the loop $\beta$ are contained in a disk with three punctures $D_3$. Let $z_1, z_2$ be as in Lemma \ref{pmtwist}. We then have that $T_{\beta}\cdot y = [z_1, z_2]y.$
     
   But for general  $g_1, g_2,g_3 \in G$ we can find a representation $\rho$ such that $\rho(z_1) =g_1, \rho(z_2) =g_2$ and $\rho(y) = g_3.$ This can be seen as follows: It is possible to define such a representation $\rho$ on the subsurface $\Sigma_2$ which is a thickening of $S \cup \beta \cup y$. We then extend this representation to $\Sigma^+,$ since the complementary subsurface has high enough genus, this can always be done.
   
   This then implies that $\bar f$ is invariant under left multiplication by $[G,G] = G$, which in turn implies that $\bar f$ and $f$ are constant maps.
    
\end{proof}

\begin{lemma}\label{indpmergodic}
Let $\Sigma$ be a surface of genus $\ge 2$ with a single boundary component $b_0$. Let $T$ be a two holed torus. We glue $T$ to $\Sigma$ along $b_0$ to obtain a surface $\Sigma^+$. The remaining boundary component of $\Sigma^+$  will be denoted $b$.
Let $\alpha, \beta$ be two curves in $\Sigma^+$ that are as in the statement of Lemma \ref{indpm}.
Given $k \in K$, we let $\Rc_K(\Sigma^+, *)$ be as in subsection \ref{gluetorerg}.

Then for all $k \in K$,
 $<\Mod(\Sigma), T_{\alpha}, T_{\beta}>$ acts ergodically on $\Rc_k(\Sigma^+, *).$

\end{lemma}
\begin{proof}
    This follows mutatis mutandis from the proof of the previous lemma.
\end{proof}
\section{Base cases and entire mapping class group}

\subsection{Once punctured torus}
Let $\Sigma$ be a torus with one boundary component $b$. Let $* \in b.$
Recall that given $g \in G$, $\Rc_g(\Sigma, *)$ is the space of representations  $\rho: \pi_1(\Sigma, *)  \to G$ such that $\rho(b) =g.$
\begin{lemma}\label{Basealg}
    For general $g \in G$,    $\Mod(\Sigma)$ acts with Zariski dense orbits on $\Rc_g(\Sigma, *)$
\end{lemma}

Before embarking on the proof of this Lemma, let us state an analogous result proven by Pickrell-Xia. Given $k \in K,$ let  $\Rc_{k,K}(\Sigma, *)$ be the space of representations  $\rho: \pi_1(\Sigma, *)  \to K$ such that $\rho(b) =k.$

\begin{lemma}[Theorem 2.14 of \cite{PX}] \label{basePX}
    For $a.e.$ $ k \in K$ the group $\Mod(\Sigma)$ acts ergodically on $\Rc_{k,K}(\Sigma, *).$ 
\end{lemma}

Let us now prove Lemma \ref{Basealg}.
\begin{proof}
    Let $K \subseteq G$ be a maximal compact subgroup.
    We will prove this theorem in two steps, we will first prove that for a.e. $k \in K$ (with respect to the Haar measure on $K$), $\Mod(\Sigma)$ acts with Zariski dense orbits on $\Rc_k(\Sigma).$ We will then show that this implies the Lemma.

    Let us note that we can view $\Rc_{k,K}(\Sigma, *)$ as a Zariski dense subset of $\Rc_k(\Sigma).$ Furthermore, for general $k \in K$ the action of $\Mod(\Sigma)$ on $\Rc_{k,K}(\Sigma, *)$ is ergodic and hence possesses a point $\rho$ whose orbit is dense in the Euclidean topology. But then the Zariski closure of $\Mod(\Sigma) \cdot \rho$ contains $\Rc_{k,K}(\Sigma, *)$ and hence the closure must be all of $\Rc_k(\Sigma, *).$

    Now we will establish that the Lemma holds for all $g \in G$ (not just $k \in K \subseteq G$). We note that either one of the following conditions must be true:
    \begin{enumerate}
        \item For very general $g \in G$, $\CC (Rc_g(\Sigma^+))^{\Mod(\Sigma)} = \CC.$
        \item  For very general $g \in G$, $\CC (Rc_g(\Sigma^+))^{\Mod(\Sigma)} \neq \CC.$
    \end{enumerate}
    However, the fact that for a.e. $k \in K$ we have $\CC (Rc_g(\Sigma^+))^{\Mod(\Sigma)} = \CC$ is true rules out the second possibility, since $K \subseteq G$ is not contained in any divisor, it is Zariski dense.
\end{proof}

We note that the above could also be proven more directly, i.e. without appealing to Lemma \ref{basePX}, by mimicking the proof of Theorem 2.14 \ref{basePX} in \cite{PX}. But this would be somewhat more involved. 
\subsection{Framed mapping class groups}

In this section we will prove that a framed mapping class group acts with Zariski dense orbits (or ergodically) on relative character varieties. Let us first recall what framed mapping class groups are.

There are several ways of defining framed mapping class groups. For us we will choose a winding number function approach. We suggest reading \cite{CS} section 2 of \cite{CS} for more details.
Let $\Sigma$ be a Riemann surface. Let $C(\Sigma)$ be the set of oriented simple closed curves in $\Sigma$.
Let $\phi: C(\Sigma) \to \ZZ$ be a function satisfying:
\begin{enumerate}
    \item $\phi(T_a \cdot x) = \phi(x) + <x,a>\phi(a),$ (where $<>$ is the algebraic intersection number).
    \item Let $S \subseteq C$ be a subsurface with $k$ boundary components $a_1, \dots, a_k,$ oriented so that $S$ lies to the left of each $a_i$. Then, $\phi(a_1) + \phi(a_2) + \dots + \phi(a_k) = \chi(S).$
\end{enumerate}

Such a function is called a winding number function. The mapping class group $\Mod(\Sigma)$ acts on the set of such  winding number functions. Given a winding number function $\phi$, let $\Mod(\Sigma)[\phi]$ denote the stabiliser of $\phi$. We will call such a stabiliser subgroup a framed mapping class group.

\begin{definition}
     Let $\alpha, \beta$ be disjoint simple closed curves. Let $a$ be an arc joining them together, such that the interior of $a$ is disjoint from $\alpha \cup \beta,$ and such that $a$ approached both $\alpha$ and $\beta$ from the left. Let  $\alpha \cup_a \beta $ be the loop formed by following $\alpha$ going through the arc $a,$ following $\beta$ and coming back through $a$. The loop $\alpha \cup_a \beta $  can be perturbed to a simple closed curve.
\end{definition}

We note that $\alpha \cup_a \beta $, $\alpha$ and $\beta$ together bound a pair of pants $P$. We orient  $\alpha \cup_a \beta $ such that $P$ is on the right. We then note that $\phi (\alpha\cup_a \beta) = \phi(\alpha) + \phi(\beta) -1.$

In this section, we will need to deal with relative representation varieties for surfaces with boundary quite often. Our convention is as follows. Let $\Sigma$ be a surface with boundary and $*$ a point on a distinguished boundary component $b$. Let the other boundary components be $\delta_1, \dots \delta_n$ ($n$ is allowed to be zero). Our boundary data $\partial$ will consist of $n$ conjugacy classes $d_1, \dots d_n$ in $G$ (resp. $K$). Let  $g \in G$ (resp. $K$). We then define $\Rc^\partial_g(\Sigma, *)$ (resp. $\Rc_{g,K}^\partial(\Sigma, *)$ ) to be the space of representations $\rho$ of $\pi_1(\Sigma, *) \to G$ (resp. $K$) such that $\rho(b) =g$ and $\rho(\delta_i) \in d_i$.

\begin{lemma}\label{framedense}
    Let $\Sigma$ be a surface with boundary of genus $ \ge 4$. Let $*$ denote a point on the boundary. Let $\partial$ denote a choice of boundary data. Let $g \in G$. Let $\phi$ be a framed winding number function. Then the $\Mod(\Sigma)[\phi]$ action on  $\Rc^{\partial}_g(\Sigma, *)$ admits a Zariski dense orbit.
\end{lemma}
\begin{proof}
It suffices to prove that the closure of $\Mod(\Sigma)[\phi]$ with respect to the action on $\Rc^{\partial}_g(\Sigma, *)$ has a Zariski dense orbit.
Our proof will be by iteratively applying Lemmas \ref{indtor} and \ref{indsamebound} . We will produce a series of surfaces $\Sigma_0 \subseteq  \Sigma_1 \dots \subseteq \Sigma_N \subseteq \Sigma$ satisfying the following.
\begin{enumerate}
   \item $\Sigma_0$ is a torus with one boundary component, such that the image of $\Mod(\Sigma_0)$ is contained in $\Mod(\Sigma)[\phi].$
    \item $\Sigma_N$ is a subsurface of the same genus as $\Sigma$.
    \item $\Sigma_{i+1}$ is obtained by attaching a two holed torus $T_i$ to $\Sigma_i$ and furthermore,
    \item there are simple closed curves $\alpha_i, \beta_i \subseteq \Sigma$ satisfying $\phi(\alpha_i) = \phi(\beta_i) =0,$ such that a thickening of $\Sigma_i \cup \alpha_i \cup \beta_i$ is $\Sigma_{i+1}.$
    \item The curves $\alpha_i$ and $\beta_i$ are in appropriate position to apply Lemma \ref{gluetor}.

\end{enumerate}
Let us note that if we could do this, we could iteratively apply Lemmas \ref{indtor} and \ref{changeofclosurealg} to conclude that  the closure of the subgroup $$<\Mod(\Sigma_0, T_{\alpha_1}, T_{\beta_1}, \dots T_{\alpha_N}, T_{\beta_N})>$$ contains $\Mod(\Sigma_N).$

 We note that for any choice of framing $\phi$, the first of the  properties in the list can always be satisfied, we can find some $\Sigma_0$ such that $\Mod(\Sigma_0) \subseteq \Mod(\Sigma)[\phi].$

 Similarly, the other conditions can always be satisfied, we only have to produce at each stage curves $\alpha_i, \beta_i$ with framed winding number zero and prescribed intersection patterns. To produce these curves we proceed as follows. We first construct some curve $\alpha$ whose intersection with $\Sigma_i$ is a single arc. If $\Sigma \setminus \Sigma_i$ is of genus $\ge 2$ there is a curve $\alpha_i$ such that $\alpha \cap \Sigma_i  =\alpha_i \cap \Sigma_i$ and $\phi (\alpha_i) =0,$ we may produce this $\alpha_i$ by forming the arc connect sum of $\alpha$ with some appropriate curve $c \subseteq \Sigma \setminus \Sigma_i,$ such a $c$ necessarily exists because of our genus assumption. If $\Sigma_i$ is of genus $\ge 2$ we can produce $\alpha_i$ by arc connect summing with a curve $c \subseteq \Sigma_i.$ Once we produce $\alpha_i$ we can then produce  some $\beta$ whose intersection with a thickening of $\Sigma_i \cup \alpha_i$ is a single arc entering throug one boundary component and leaving through the other. We can then modify $\beta,$ just as above to obtain the require curve $\beta_i$ such that $\phi(\beta_i)=0.$
 

 This allows us to inductively establish that $\overline{\Mod(\Sigma)[\phi]} \cap \Mod(\Sigma_N)$ acts with Zariski dense orbits on the varieties $\Rc_g(\Sigma_N, *_N)$ for general $g$, and thus $\Mod(\Sigma_N) \subseteq \overline{\Mod(\Sigma)[\phi]}.$

 Finally, we find strips $a_j$ in $\Sigma$ with boundary on $\Sigma_N$, such that attaching these arcs to $\Sigma_N$ and thickening them gives us $\Sigma$. This can always be accomplished since $\Sigma_N$ and $\Sigma$ are of the same genus. We then find loops $\gamma_j$ such that $\gamma_j \cap \Sigma \setminus \Sigma_N  = a_j,$ and $\phi(\gamma_j ) =0.$ Again, this can always be accomplished. We have that $T_{\gamma_i} \in \Mod(\Sigma)[\phi].$ We can then apply Lemma \ref{indsamebound} repeatedly to conclude that $\overline{\Mod(\Sigma)[\phi]}$ acts with Zariski dense  orbits on relative representation varieties of subsurfaces containing $\Sigma_N$ with more and more  of the boundary components  of $\Sigma$ until we finally conclude that $\overline{\Mod(\Sigma) [\phi]}$ acts with Zariski dense orbits on $\Rc^{\partial}_g(\Sigma, *).$

\end{proof}
\begin{lemma}\label{framedenseerg}
    Let $\Sigma$ be a surface with boundary of genus $ \ge 2$. Let $*$ denote a point on the boundary. Let $\partial$ denote a choice of boundary data. Let $\phi$ be a framed winding number function. Then the $\Mod(\Sigma)[\phi]$ action on  $\Rc^{\partial}_K(\Sigma, *)$ is ergodic.
\end{lemma}
\begin{proof}
    This follows mutatis mutandis from the proof of the above Lemma.
\end{proof}

We note that the above two lemmas work even in the case where $\Sigma$ is a surface with one boundary component and the monodromy around the puncture is trivial. Thus we have:
\begin{lemma}
    Let $\Sigma$ be a surface with one boundary component $b$. Let $* \in b$. Let $\bar\Sigma$ be the surface obtained by capping off the boundary with a disk.
    Then:
    \begin{enumerate}
        \item $\Mod(\Sigma)$ acts with Zariski dense orbits on $\Rc_1(\Sigma, *) = \Rc(\bar \Sigma, *).$
        \item $\Mod(\bar \Sigma, *)$ acts with Zariski dense orbits on $\Rc(\bar \Sigma, *)$.
        \item $\Mod(\Sigma)$ acts ergodically on $\Rc_{1,K}(\Sigma, *) = \Rc_K(\bar \Sigma, *).$
        \item   $\Mod(\bar \Sigma, *)$ acts ergodically on $\Rc_K(\bar \Sigma, *)$.
    \end{enumerate} 
\end{lemma}
\begin{proof}
    Statements (2) and (4) immediately follow from statements (1) and (3). Statements (1) and (3) in  turn follow from the fact that $\Mod(\Sigma)$ contains a framed mapping class group $\Mod(\Sigma)[\phi]$ (for any  choice of framing) and the previous two lemmas.
\end{proof}

\begin{lemma}\label{rempunct}
    Let $\Sigma$ be a closed surface, with a basepoint $*$. Let $\Gamma \subseteq \Mod(\Sigma,*) $ be a group that acts with Zariski dense orbits on $\Rc(\Sigma, *).$ Then the image of $\Gamma$ in $\Mod(\Sigma)$ acts with Zariski dense orbits on $\XX(\Sigma).$
\end{lemma}
\begin{proof}
    Let $p: \Rc(\Sigma,*) \to \XX(\Sigma)$ denote the projection. Let $f \in \CC(\XX(\Sigma))$ be a rational function invariant under the image of $\Gamma.$ Then $f \circ p$ is a $\Gamma$ invariant rational function on $\Rc(\Sigma, *)$ and hence is constant. Thus $f$ is also constant. This establishes the lemma.
\end{proof}

\begin{lemma}\label{rempuncterg}
    Let $\Sigma$ be a closed surface, with a basepoint $*$. Let $\Gamma \subseteq \Mod(\Sigma,*) $ be a group that acts with Zariski dense orbits on $\Rc_K(\Sigma, *).$ Then the image of $\Gamma$ in $\Mod(\Sigma)$ acts ergodically on $\XX_K(\Sigma).$
\end{lemma}
\begin{proof}
    This follows mutatis mutandis from the proof of the previous lemma.
\end{proof}

\begin{lemma}\label{wholemcg}
    Let $\Sigma$ be a closed surface with base point $*$. The action of $\Mod(\Sigma)$ on $\XX_K(\Sigma)$ is ergodic and the action of $\Mod(\Sigma)$ on $\XX(\Sigma)$ has Zariski dense orbits.
\end{lemma}
\begin{proof}
    The preimage of $\Mod(\Sigma)$ in $\Mod(\Sigma, *)$ is all of $\Mod(\Sigma, *),$ and $\Mod(\Sigma,*)$ acts with Zariski dense orbits on $\Rc(\Sigma, *)$ and acts ergodically on $\Rc_K(\Sigma, *).$  Thus by Lemmas \ref{rempunct} and \ref{rempuncterg} the result folows.
\end{proof}
\section{Producing Dehn twists}

Let us recall some of the techniques used in \cite{BS}  to understand the monodromy groups associated to complete linear systems. While in that paper the authors were dealing with curves in a simply connected surface $X$, many of the ideas and constructions involved do not require that hypothesis.
\subsection{Set up}
Let $X$ be an algebraic surface. Let $\Lc$ be a very ample line bundle on $X$. Let $d >0$ be a positive integer. Let $$U(X, \Lc^d)  = \{f \in H^0(X, \Lc^d)| Z(f) \textrm{ is smooth}\}.$$ Let $$E(X, \Lc^d) =\{ (x,f) \in X \times U(X, \Lc^d)| f(x) =0\}.$$The map $E(X ,\Lc^d) \to U(X, \Lc^d)$ is a fiber bundle. We will denote the monodromy group of this fiber bundle by $\Gamma$.

Strictly speaking, to even discuss the monodromy group of a fiber bundle we need a base point. We will pick a very special base point to help our computations. Let $d_1, d_2$ be positive integers such that $d_1 + d_2 = d.$ Let $f \in U(X, \Lc^d_1)$ and $g \in U(X, \Lc^{d_2})$ be such that $C = Z(f)  $ and $D = Z(g)$ intersect transversely. In this situation $Z(fg) = Z(f) \cup Z(g) = C \cup D$ is the union of two algebraic curves meeting transversely. This union is not smooth however, it has nodal singularities at each point of $C \cup D.$ We then consider a small perturbation of $fg$, of the form $fg + \epsilon h,$ where $\epsilon>0$ is suitably small and $h$ is suitably generic. We let $E = Z(fg + \epsilon h).$ It is this curve $E$ that will be our base point, and $\Gamma$ will be a subgroup of $\Mod(E).$

Let $\tilde C, \tilde D$ denote the real oriented blow ups of $C$ and $D$ along the locus $C \cap D.$

In \cite{BS}, the authors prove that the curve $E$ is homeomorphic to the Riemann surface formed by gluing $\tilde C $ and $\tilde D$ along the boundary circles formed in the blowup. This homeomorphism is also canonical up to a mild ambiguity. The reader is advised to see Section 3.4 of \cite{BS} for more details and several helpful figures. 

\subsection{Deforming Singularities}
Let us now describe some techniques used by the authors of \cite{BS} to produce Dehn twists in $\Gamma.$ Let us first recall the notion of the vanishing cycle associated to a nodal degeneration. Let $p:\CCa \to \DD$ denote a proper homolomorphic family of curves over the unit disk. Let us assume that $\CCa$ is smooth and that the map $p$ is a submersion at all but one point $x$, where $p(x) =0$. Let $\CCa^* \to \DD \setminus\{0\}$ denote the pull back of $p$ to $\DD \setminus \{0\},$ it is a fiber bundle by Ehresmann's theorem. 

Assume that the fiber over $0$, $\CCa_0$  is a nodal curve with a single node at $x$. Let $t \neq 0.$  In this situation, the nodal fiber $\CCa_0$ is homeomorphic to a quotient of a smooth fiber, $\CCa_t$ after we collapse a certain simple closed curve $\alpha$ in $\CCa_t$ to a point . This curve $\alpha$ is called the \emph{vanishing cycle} associated to the nodal degeneration. Let us furthermore assume that the map $p$ defines a Lefschetz fibration. In this case the monodromy  map $\pi_1(\DD \setminus\{0\},t) \to \Mod(\CCa_t)$ is precisely the map sending the generator of $\pi_1(\DD \setminus\{0\},t)$ to $T_{\alpha}$. 

Similar statements to the above hold for more complicated singularities. Let  $\CCa \to B$ be a family of curves and let $* \in B$ be a base point. Assume that $\CCa_*$ has an isolated singularity at a point $x$ and that the family is a versal deformation space for this singularity. In this case, let $B_0 \subseteq B$ denote the locus where the fibers are smooth. We again have a monodromy map $\pi_1(B_0, *) \to \Mod(\CCa_*).$ The image of this map has beens studied \cite{PCS} and is precisely a framed mapping class group of an embedding of a versal deformation of the singularity.

In \cite{BS} the authors consider various controlled degenerations of $E$ to singular curves, argue that $\Gamma$ contain the associated monodromy groups and then use various techniques to argue that $\Gamma$ is a finite index subgroup of $\Mod(E).$
\subsection{Tacnodal degenerations and vanishing cycles}
An important class of degenerations that we will consider is a tacnodal degeneration. Our base curve $E$ is obtained as a small perturbation of the union of two curves $C$ and $D$ meeting transversally. Let us consider a family of curves $C_t$ (for $t \in \DD$) , such that:
\begin{enumerate}
    \item $C_1 =C$, our original curve.
    \item For $t \neq 0,$ $C_t$ intersects $D$ transversely.
    \item The intersection of $C_0$ and $D$ is transverse except at one point $x$, where the intersection is of order $2.$
\end{enumerate}

The singularity of $C_0 \cup D$ at $x$ is called a tacnode. One can perturb the above family to get a degeneration of $E$ to a tacnodal curve. This degeneration in turn gives us certain elements in $\Gamma$ that are useful for monodromy computations.

\begin{definition}[Curve of tacnodal type]
    Let $\alpha \subseteq E$. We say $\alpha$ is of tacnodal type if $\alpha \cap \tilde C$ and $\alpha \cap \tilde D$ are single arcs and the end points of these arcs are on different boundary components of $\tilde C$ and $\tilde D$ respectively.
\end{definition}

\begin{definition} [Tacnodal vanishing cycle]
    Let $\alpha \subseteq E$ be a simple closed curve. We say that $\alpha$ is a tacnodal vanishing cycle if $T_{\alpha} \in \Gamma $ and $\alpha$ is of tacnodal type.
\end{definition}

\subsection{Fixing the right hand side of the tacnodal degeneration}
For the rest of this section we adopt the following conventions.
    \begin{enumerate}
        \item Let $\overline{\Gamma}^{K}$ denote the closure of $\Gamma$ with respect to the action of $\Mod(E)$ on $\XX_K(E)$ (a measure preserving action)
        \item Let $\overline{\Gamma}$ denote the closure of $\Gamma$ with respect to the action of $\Mod(E)$ on $\XX_G(E)$ (an algebraic action).
    \end{enumerate}

Let us first recall the Birman exact sequence as it would be helpful for our understanding of some of the subsequent material. Let $\tilde \Sigma$ be a surface with $n$ boundary components. Let $\Sigma_0$ be the punctured surface obtained by filling in the boundary components with punctured disks. Let $\Sigma$ denote the closed surface obtained by filling in these punctures.

We have surjective maps $\Mod(\tilde \Sigma) \to \Mod(\Sigma_0)$ and $\Mod(\Sigma_0) \to \Mod(\Sigma).$ Let us describe the kernels of these maps. The kernel of the map $\Mod(\Sigma_0) \to \Mod(\Sigma)$ is the \emph{point pushing subgroup} to be precise, we have a short exact sequence
$$1 \to Br_n(\Sigma) \to \Mod(\Sigma_0) \to \Mod(\Sigma) \to 1,$$ where $Br_n(\Sigma)$ is the $n$ stranded braid group on the surface $\Sigma$ and the map $Br_n(\Sigma) \to \Mod(\Sigma_0)$ is given by point pushing.
we also have a short exact sequence $$1 \to \ZZ^n \to \Mod(\tilde \Sigma) \to \Mod(\Sigma_0) \to 1,$$ here the map $\ZZ^n \to \Mod(\tilde \Sigma) $ is given by sending the generators of $\ZZ^n$ to Dehn twists about the boundary components. We define $Br(\tilde \Sigma)$ to be the kernel of $\Mod(\tilde \Sigma) \to \Mod(\Sigma).$

We note that we have a surjective map $Br_n(\Sigma) \to S_n$ (the symmetric group on $n$ letters) given by following the punctures to their destinations. We call the kernel of this map $PBr_n(\Sigma),$ the pure $n-$stranded surface braid group.

Let $M_1 \subseteq \Mod(E)$ denote the subgroup of $\Mod(E)$ that preserves $\partial \tilde C  = \partial \tilde D,$ as a set. We have a morphism $M_1 \to \Mod(C) \times \Mod(D),$ let $M$ denote the kernel of this map. The group $M$ in turn admits a map to $Br(\tilde C).$ 

Recall that if $\tilde C$ has $n$ boundary components then we have a surjective map $Br(\tilde C) \to Br_n(C)$. We further have a map $Br_n(C) \to H_1(C, \ZZ).$ This last map may be described geometrically as follows: A braid $\beta$ in $C$ traces out a set of loops in $C$, we define the image of $\beta$ to be the sum of the homology classes of these loops. The kernel of this map $Br_n(C) \to H_1(C,\ZZ)$ is called the \emph{simple braid group}. We will denote it by $SBr_n(C).$ We will use $SBr(\tilde C)$ to denote its preimage in $Br(\tilde C).$  We define $$PSBr_n(C) = SBr_n(C) \cap PBr_n(C).$$ We will call its preimage $PSBr(\tilde C).$

\begin{lemma}\label{imageSBr}
    The image of $\Gamma \cap M \to Br(\tilde C)$ contains $SBr(\tilde C)$.
\end{lemma}

This is essentially Proposition 7.1 of \cite{BSCI}, though that is stated in  a somewhat limited context. This was used again in \cite{BS} in the more general context where our ambient surface $X$ is simply connected. However simply-connectedness is never used in the proof of this Lemma.

Let us mention a few consequences of this Lemma.

\begin{lemma}\label{normal}
    The group $\Gamma \cap \Mod(\tilde C)$ is normal under $SBr(\tilde C)$.
    The groups $\overline{\Gamma} \cap \Mod(\tilde C)$ and $\overline{\Gamma}^K \cap \Mod(\tilde C)$ are also normal under $SBr(\tilde C)$.
\end{lemma}
\begin{proof}
     Let $g \in SBr(\tilde C).$ Let $\tilde g \in \Gamma \cap M$ be an element mapping to $g$. Let $h \in  \Gamma \cap \Mod(\tilde C).$ We note that $\tilde g h \tilde g^{-1} = ghg^{-1},$ and that $\tilde g h \tilde g^{-1} \in \Gamma \cap \Mod(\tilde C).$ The same argument works for $\overline{\Gamma} \cap \Mod(\tilde C)$ and $\overline{\Gamma}^K \cap \Mod(\tilde C).$ This establishes the Lemma. 
\end{proof}
\begin{lemma}
    Let $a$ be an arc in $\tilde C$ joining two boundary componenets $\Delta_i$ and $\Delta_j$. Then there is a tacnodal vanishing cycle $\alpha$ such that $\alpha \cap \tilde C = a.$
\end{lemma}
\begin{proof}
    Let $\alpha_0$ be any tacnodal vanishing cycle going between $\Delta_i$ and $\Delta_j$. Let $a_0 = \alpha_0 \cap \tilde C$. We note that there is some $g \in SBr(\tilde C)$ such that $g \cdot a_0 = a.$ Let $ \tilde g \in \Gamma \cap M$ be an element that maps to $g$ under the obvious map. Then $ \alpha =\tilde g \cdot \alpha_0$ is a tacnodal vanishing cycle. Furthermore $ \alpha \cap \tilde C = g \cdot a_0 =a$.
\end{proof}
We would like to emphasise that if $\pi_1X \neq 0$, it is not actually possible for every curve of tacnodal type to be a tacnodal vanishing cycle. Indeed if  the simple closed curve $\alpha \subseteq E$ is not nullhomotopic in $X$, then it can not be a vanishing cycle. 

\begin{lemma}
    The monodromy group $\Gamma$ contains a framed mapping class group of a subsurface $\Sigma_0$ with one boundary component such that:
    \begin{enumerate}
        \item $\Sigma_0$ has genus $\ge 4$.
        \item $\Sigma_0$ contains three of the boundary components of $\tilde C$ (note that these won't be boundary components of $\Sigma_0$).
        \item $\Sigma_0$ contains a  loop $\alpha$ obtained by performing an   arc connect sum of two of the boundary components $\Delta_i$ and $\Delta_j$ along an arc entirely in $\tilde C.$

    \end{enumerate}
  Furthermore, $\overline{\Gamma}$ and $\overline{\Gamma}^K$ contain $T_{\alpha}.$
\end{lemma}
\begin{proof}
    It suffices to establish that $\Gamma$ contains a framed mapping class group of a subsurface $\Sigma'$ that contains a subsurface $\Sigma_0$ of the desired type. However this follows from Lemma 6.1 of \cite{BS} where the authors construct a 'core subsurface' that satisfies the conditions we desire for $\Sigma'$ (that section also does not use simply connectedness).
    
    By Lemmas \ref{framedense}, \ref{framedenseerg}   $\Mod(\Sigma_0) \subseteq \overline{\Gamma} , \overline{\Gamma}^K$. Since $T_{\alpha} \in \Mod(\Sigma_0)$, we have $T_{\alpha} \in \overline{\Gamma}, \overline{\Gamma}^K.$

\end{proof}

\begin{lemma}\label{containsker}

    The groups $\overline{\Gamma}, \overline{\Gamma}^K$ contain the kernel of $PSBr(\tilde C) \to \pi_1(C)^n.$
\end{lemma}
\begin{proof}
    We note that the orbit of $T_{\alpha}$ under $SBr(\tilde C)$ generates the kernel of $PSBr(\tilde C) \to \pi_1(C)^n$ (see for instance, the proof of Lemma 7.2 of \cite{BS}).  Since $\bar\Gamma\cap M$ surjects onto $SBr(\tilde C)$, $\bar \Gamma$ in fact contains this $SBr(\tilde C)$ orbit of $T_{\alpha}$. The same applies to $\overline{\Gamma}^K.$
\end{proof}

\begin{lemma}\label{extarc}
Let $\Sigma_0$ be a subsurface of $E$ such that the genus of $\Sigma_0 \cap \tilde C$ is at least two. Assume that for some choice of framing $\phi$ on $\Sigma_0$,  $\Gamma$ contains $\Mod(\Sigma_0)[\phi].$ We furthermore assume $\Sigma_0$ contains three boundary components $\Delta_1, \Delta_2$ and $\Delta_3$ and contains a vanishing cycle $\alpha$ joining the first two components.

Let $a$ be an arc whose interior lies in $\tilde C \setminus \Sigma_0$ and whose end points lie on the boundary of $\tilde C \setminus \Sigma_0$. Then $\overline \Gamma, \overline{\Gamma}^K$ contain $\Mod(\Sigma)$, where $\Sigma$ is a thickening of $\Sigma_0 \cup a$. 
\end{lemma}
\begin{proof}
    We note that $\Sigma= \Sigma_0 \cup S$ where $S$ is a 1-handle that is parallel to $a$.
    It suffices to establish that $\bar \Gamma$ contains a Dehn twist about a loop $\beta$ in $\pm 1$ position with respect to $(\Sigma_0, S)$.  Let us produce such a $\beta$. The curve $\beta$ that we will will be of the form $g \cdot \alpha$ for some $g \in PSBr(\tilde C) \cap \overline{\Gamma}$  it is immediate that any $\beta$ of this form is such that $T_{\beta} \in \overline{\Gamma}.$
    The tacnodal vanishing cycle $\alpha$ goes through two boundary components $\Delta_1$ and $\Delta_2.$ 
    Let $l$ denote a path joining $\Delta_1$ to an endpoint of $a$.
    Let $\bar a$ be a loop in $\tilde C$ starting at $\Delta_1$, obtained by:
    \begin{enumerate}
        \item Following $l$ till we  reach $a$
        \item Following $a$ till we reach the other endpoint
        \item Returning to the first endpoint of $a$  via some path in $\Sigma_0.$
        \item Going back via $l$.
    \end{enumerate}
    Let $\gamma$ be the braid obtained by point pushing our first boundary  $\Delta_1$ along  $\bar a.$
    Let $\sigma$ be a full arc twist about an arc joining $\Delta_1$ and $\Delta_3$ disjoint from $\alpha$.
 
    We claim that $g = [\gamma, \sigma]$ satisfies our requirements.
    To see that $\beta = g \cdot \alpha$ is in $\pm 1$ position follows from drawing out the loop $\beta$. It also follows from the following two facts (which are easy to see from the defintion), the loop $\beta$ is not isotopic to a loop in $\Sigma_0$, and it enters the  attaching arc once in the positive direction and once in the negative direction. It is also to easy to see that the complementary region is connected. Now to establish that $g \in \overline {\Gamma}$  (resp. $\overline{\Gamma}^K$)it suffices by Lemma \ref{containsker} to establish that the image of $g$ in $\pi_1(C)^n$  is in the image of $\overline{\Gamma} \cap PSBr(\tilde C)$ (resp. $\overline{\Gamma}^K$)in $\pi_1(C)^n.$ But this follows from the fact that $\sigma$ and hence $[\gamma,\sigma]$ is mapped to the identity by this map.
\end{proof}

\begin{lemma}\label{contnormalpha}
 The groups $\overline{\Gamma}$ and $\overline{\Gamma}^K$ contain Dehn twists $T_{\alpha_1} ,\dots T_{\alpha_N}$ such that $\alpha_1 ,\dots , \alpha_N \subseteq \tilde C$ normally generate $\pi_1(C)$. 
\end{lemma}
\begin{proof}
    Let $\Sigma_0$ be as in the statement of Lemma \ref{extarc}. Let $a_1 ,\dots a_M$ be a collection of arcs in $\tilde C $  such that:
    \begin{enumerate}
        \item The end points of $a_i$ lie on the same boundary component of $\Sigma_0$.
        \item The interiors of the $a_i$ are disjoint from $\Sigma_0.$
        \item The union of the images of $\pi_1((\Sigma_0 \cup a_i)\cap \tilde C) \to \pi_1(C)$ normally generate $\pi_1(C),$ 
    \end{enumerate}
    This can always be arranged.  We then pick simple closed curves $\alpha_i^M \subseteq (\Sigma_0 \cup a_i)\cap \tilde C) $ that normally generate $\pi_1((\Sigma_0 \cup a_i)\cap \tilde C)$. By Lemmas \ref{extarc}, \ref{changeofclosurealg}, and \ref{changeofclosureerg} the $T_{\alpha_i^M}$ are contained in $\overline{\Gamma}$ and $\overline{\Gamma}^K.$

\end{proof}
\begin{lemma}\label{normgen}
    Let $\{k_i\}_{i\in I}$ be a set of normal generators for $\pi_1(C)$.
    Let $N_0 \subseteq \pi_1(C)^n$ be the subgroup   generated by elements of the form $$(1 , \dots 1 , k_i, k_i^{-1}, 1 ,\dots 1)$$ (here $k_i$ is in the $j^{th}$ position for some $1 \le j < n$). Let $N$ be the normal closure of $N_0$ by elements of $\pi_1(C)^n_0.$

    Then $N = \pi_1(C)^n_0$.
\end{lemma}
\begin{proof}
We note that given a generator $g$ of $N_0$, the set $\pi_1(C)^n_0 \cdot g = \pi_1(C)^n \cdot g$, i.e. any conjugate of $g$ in the bigger group can be realised by conjugation with the smaller group. This immediately follows from the fact that the centraliser of $g$, $C(g)$ surjects onto $H_1(C, \ZZ)$ under the composite map $C(g) \to \pi_1(C)^n \to H_1(C,\ZZ).$
This establishes that $N$ is normal in $\pi_1(C)^n.$

We claim that for any $g \in \pi_1(C)$ the element $$(1 , \dots 1, g, g^{-1}, 1 , \dots 1)\in N$$ (here $g$ is the $j^{th}$ position). This follows from the following argument. The set of $g$ for which the claim holds is conjugation invariant and contains the generators $k_i$. It suffices to prove that this set is closed under multiplication.  Let $x_1, x_2 \in \pi_1(C)$ be such that $$(1 ,\dots 1 , x_i , x_i^{-1} , 1, \dots, 1) \in N.$$ Then we immediately get that $$(1 , \dots 1, x_1 x_2 , x_1^{-1} x_2^{-1}, 1 , \dots 1) \in N.$$ But since $N$ is normal, $$(1 , \dots 1, x_1 x_2 , (x_1x_2)^{-1}, 1 , \dots 1) \in N$$ as well.

Now the fact that we have all elements of the form  $$(1 , \dots 1, g, g^{-1}, 1 , \dots 1)$$ in $N$ allows us to conclude that every element of $\pi_1(C)^n / N$ has a representative of the form $(1 , \dots 1 , g).$
We note that this argument also implies that every element in $\pi_1(C)^n/N$ has a representative of the form  $(g, 1 ,\dots,1).$ This in turn tells us that the quotient group is abelian and  that $([g,h] ,1 \dots , 1) \in N.$ This in turn tells us that the map $\pi_1(C)^n /N \to H_1(C)$ is an isomorphism. This implies that $N = \pi_1(C)^n_0.$
\end{proof}

\begin{lemma}\label{contpsbr}
 The maps  $\overline{\Gamma} \cap PSBr(\tilde C) \to \pi_1(C)^n_0$  and $\overline{\Gamma}^K \cap PSBr(\tilde C) \to \pi_1(C)^n_0$ are surjections. As a result $\overline{\Gamma}, \overline{\Gamma}^K$ contain $PSBr(\tilde C).$
\end{lemma}
\begin{proof}
     We use Lemma \ref{contnormalpha} to get a collection of Dehn twists $T_{\alpha_i} \in \overline{\Gamma} $ such that the $\alpha_i$ normally generate $\pi_1(C)$.  We will abuse notation and let $\alpha_i$ denote an element in $\pi_1(C)$ belonging to this conjugacy class.
    
    Consider an arc $a$ joining the two boundary component $\Delta_j$ and $\Delta_{j+1}.$ Let $\sigma$ denote the arc half twist about $a$. By Lemma \ref{normal}, $[\sigma, T_{\alpha_i}] \in \overline {\Gamma} \cap PSBr(\tilde C).$ However this braid is sent to an element conjugate to $(1, \dots,\alpha_i, \alpha_i^{-1} )$  in $\pi_1(C)^n.$ Thus, by Lemma \ref{normgen}, the image of the set of all such $[\sigma, T_{\alpha_i}]$ normally generate $\pi_1(C)^n_0$. Thus the image of $\overline{\Gamma}$ normally generates $\pi_(C)^n_0$. But the image is normal by Lemma \ref{normal}. Thus the map is a surjection. 

    The argument works equally well for $\overline{\Gamma}^K.$
\end{proof}
\begin{lemma}\label{fullclosure}
    $\overline{\Gamma} = \overline{\Gamma}^K =\Mod(E).$
\end{lemma}
\begin{proof}
    Let $\alpha$ be a curve of tacnodal type.  We claim that $T_{\alpha} \in \overline{\Gamma} , \overline{\Gamma}^K.$ We first note that there is a tacnodal vanishing cycle $\alpha_0$ such that $\alpha_0 \cap \tilde D = \alpha \cap \tilde D.$ We note that $PSBr(\tilde C)$ acts transitively on arcs with ends on two distinct fixed boundary components. To see this, we observe that by the change of coordinates principle, the group $PBr(\tilde C)$ acts transitively on the set of arcs with ends on two distinct boundary components. We can then use the fact that the stabiliser of a given such arc surjects on to $H_1(C, \ZZ)$ to conclude that $PSBr(\tilde C)$ also acts transitively.

    Thus there is some $g \in PSBr(\tilde C)$ such that $g \cdot \alpha_0 = \alpha.$ But since $ g \in \overline{\Gamma}, \overline{\Gamma}^K$ by Lemma \ref{contpsbr}, $T_{\alpha} \in\overline{\Gamma}, \overline{\Gamma}^K.$ This establishes the claim.
    
  We then finish exactly as in the proof of Theorem A of \cite{BS}. More precisely, we begin by noting that we have a subsurface $\Sigma_0$ as in Lemma \ref{extarc} , such that $\Mod(\Sigma_0) 
     \subseteq \overline{\Gamma}$. We then construct a sequence of subsurfaces $\Sigma_0 \subseteq \Sigma_1 \dots \subseteq \Sigma_N \subseteq E$, where:
     \begin{enumerate}
         \item $\Sigma_{i+1}$ is obtained by attaching an arc $a_i$ to $\Sigma_{i}.$
         \item The arc $a_i$ is such that there is a simple closed curve $\alpha_i$ such that $\alpha_i \cap (E \setminus \Sigma_i) = a_i.$ Furthermore $\alpha_i$ intersects some curve $\beta \subseteq \Sigma_i$ once tranversely. We require that $T_{\alpha_i} \in \overline{\Gamma}.$
         \item The entire surface $E$ is obtained by filling in the boundary components of   $\Sigma_N.$
     \end{enumerate}

     The construction above can always be carried out because $\alpha$ can be \emph{any} curve in $E$ of tacnodal type. 

     We now note that $<\Mod(\Sigma_i), T_{\alpha}> = \Mod(\Sigma_{i+1}).$ Thus we obtain the equality $\overline{\Gamma} = \Mod(\Sigma).$

     The above argument works equally well for $\overline{\Gamma}^K.$
\end{proof}

We finally prove the main theorem.
\begin{proof}[Proof of Theorem \ref{main}]
By Lemma \ref{fullclosure}, $\overline{\Gamma} = \Mod(\Sigma)$ and by Lemma \ref{wholemcg} $\Mod(\Sigma)$ acts with Zariski dense orbits on the character variety. Thus the monodromy group $\Gamma$ acts with dense orbits on the character variety.
\end{proof}

\begin{proof}[Proof of Theorem \ref{mainerg}]
    By Lemma \ref{fullclosure}, $\overline{\Gamma}^K = \Mod(\Sigma)$ and by Lemma \ref{wholemcg} $\Mod(\Sigma)$ acts ergodically on the character variety. Thus the monodromy group $\Gamma$ acts ergodically on the character variety.
\end{proof}

\bibliographystyle{amsalpha}
\bibliography{references}

@article{KPS,
  title={Density of monodromy actions on non-abelian cohomology},
  author={Katzarkov, Ludmil and Pantev, Tony and Simpson, Carlos},
  journal={Advances in Mathematics},
  volume={179},
  number={2},
  pages={155--204},
  year={2003},
  publisher={Elsevier}
}

@article{PCS,
  title={Vanishing cycles, plane curve singularities and framed mapping class groups},
  author={Portilla Cuadrado, Pablo and Salter, Nick},
  journal={Geometry \& Topology},
  volume={25},
  number={6},
  pages={3179--3228},
  year={2021},
  publisher={Mathematical Sciences Publishers}
}

@article{Saadi,
  title={Non-ergodicity on the SU (2)-character varieties},
  author={Saadi, Fayssal},
  journal={Commentarii Mathematici Helvetici. Email address: fayssal. saadi. math@ gmail. com},
  year={2025}
}

@article{Bpi,
  title={A $\pi_1$ obstruction to having finite index monodromy and an unusual subgroup of infinite index in $\textrm{Mod}(\Sigma_g)$},
  author={Banerjee, Ishan},
  journal={arXiv preprint arXiv:2403.07280},
  year={2024}
}

@article{BS,
  title={Monodromy and vanishing cycles for sufficiently ample linear systems on simply connected surfaces},
  author={Banerjee, Ishan and Salter, Nick},
  journal={arXiv preprint arXiv:2512.04018},
  year={2025}
}

@article{BSCI,
  title={Monodromy and vanishing cycles for complete intersection curves},
  author={Banerjee, Ishan and Salter, Nick},
  journal={arXiv preprint arXiv:2408.06479},
  year={2024}
}

@article{CS,
  title={Framed mapping class groups and the monodromy of strata of abelian differentials},
  author={Calderon, Aaron and Salter, Nick},
  journal={Journal of the European Mathematical Society},
  volume={25},
  number={12},
  pages={4719--4790},
  year={2022}
}

@article{PX,
  title={Ergodicity of mapping class group actions on representation varieties, I. Closed surfaces},
  author={Pickrell, Doug and Xia, Eugene Z},
  journal={Commentarii Mathematici Helvetici},
  volume={77},
  number={2},
  pages={339--362},
  year={2002}
}
\end{document}